\documentclass[a4paper,11pt]{article}
\usepackage[scale={0.8,0.9},centering,includeheadfoot]{geometry}

\usepackage{amsfonts}
\usepackage{amsmath}
\usepackage{amsthm}

\newtheorem{theorem}{Theorem}
  \newtheorem{lemma}[theorem]{Lemma}

  \newtheorem{definition}[theorem]{Definition}

\newtheorem{remark}[theorem]{Remark}
\newtheorem{acknowledgement}[theorem]{Acknowledgement}

\begin{document}

\title{On the convergence of the Escalator Boxcar Train}
\author{\AA{}ke Br\"annstr\"om\footnotemark[2]\ \footnotemark[3] \and Linus Carlsson\footnotemark[2]  
\and Daniel Simpson \footnotemark[4] }
\maketitle

\renewcommand{\thefootnote}{\fnsymbol{footnote}}

\footnotetext[2]{Department of Mathematics and Mathematical Statistics, SE-90187 Ume\aa{}, Sweden.}
\footnotetext[3]{Evolution and Ecology Program, International Institute for Applied Systems Analysis, A-2361 Laxenburg, Austria.}
\footnotetext[4]{Department of Mathematical Sciences, Norwegian University of Science and Technology, N-7491 Trondheim, Norway}

\renewcommand{\thefootnote}{\arabic{footnote}}

\begin{abstract}
The Escalator Boxcar Train (EBT) is a numerical method that is widely used in
theoretical biology to investigate the dynamics of physiologically structured
population models, i.e., models in which individuals differ by size or other
physiological characteristics. The method was developed more than two decades
ago, but has so far resisted attempts to give a formal proof of convergence.
Using a modern framework of measure-valued solutions, we investigate the EBT
method and show that the sequence of approximating solution measures generated
by the EBT method converges weakly to the true solution measure under weak
conditions on the growth rate, birth rate, and mortality rate. In rigorously
establishing the convergence of the EBT method, our results pave the way for
wider acceptance of the EBT method beyond theoretical biology and constitutes
an important step towards integration with established numerical schemes.
\end{abstract}



\pagestyle{myheadings} \thispagestyle{plain} \markboth{BR\"ANNSTR\"OM ET AL.}{CONVERGENCE OF THE ESCALATOR BOXCAR TRAIN}

\section{Introduction}

The population dynamics of ecological and biological systems are often
described by an ordinary differential equation of the form 
\[
\dfrac{1}{N}\frac{dN}{dt}=\beta(N)-\mu(N), 
\]
where $N=N(t)$ is the total population size at time $t$, $\beta(N)$ is the
birth rate, and $\mu(N)$ is the mortality rate, both of which depends on the
population size. The key assumption in this type of model is that every
individual in the population is identical. This is clearly unreasonable in
many situations, including cases where the gap between birth size and
reproductive size is important. A more accurate description of the
population dynamics can be given by physiologically structured population
models (see e.g., \cite{Metz1986}). In these models, the birth rates, death
rates, and growth rates of individuals depend on their physiological state $%
x\in\Omega$, where $\Omega$ is the set of admissible states. In general,
these states can represent any aspects of individual physiology such as age,
size, mass, height, or girth. For the purpose of this manuscript, we will
work with a one-dimensional state space that we think of as representing
individual size, but other interpretations are possible and, as we note in
the concluding discussion, we expect that our results can easily be extended
to higher-dimensional state manifolds.

In order to specify a physiologically structured population model, we need
explicit representations for the mortality, growth, and fecundity rates of
individuals as well as the initial population structure. We assume that
these rates are respectively on the form $\mu(x,E_t)$, $g(x,E_t)$, and $%
\beta(x,E_t)$, where $x$ is the size (or more generally the state) of the
individual and $E_t$ is the environment that individuals experiences at time $%
t$. The environment is a key factor in the formulation of physiologically
structured population models and can, for example, represent the total
amount of nutrient available at time $t$ or the size-specific predation
rate, see e.g., \cite{Metz1986,Roos1997}. While the environment is often
low-dimensional, it could potentially be infinite-dimensional as would for
example be the case for the shading profile in a forest. Finally, we assume
that all new individuals have the same birth size $x_{b}$. With these
assumptions, one can show (see e.g., \cite{Roos1997}) that the density $%
u(x,t)$ of individuals of state $x$ at time $t$ is given by the first order,
non-linear, non-local hyperbolic partial differential equations with
non-local boundary condition 
\begin{subequations}
\label{PSPM}
\begin{align}
\frac{\partial}{\partial t}u(x,t)+\frac{\partial}{\partial x}\left(
g(x,E_t)~u(x,t)\right) & =-\mu(x,E_t)u(x,t),  \label{transport}
\\
g(x_{b},E_t)u(x_{b},t) & =\int_{x_{b}}^{\infty}\beta(\xi
,E_t)u(\xi,t)\,d\xi,  \label{transport_boundary} \\
u(x,0) & =u_{0}(x),   \label{transport_initial}
\end{align}
\end{subequations}
in which we assume that $x_{b}\leq x<\infty$ and $t\geq0$.

The first numerical method designed specifically for solving physiologically
structured population models was the inventively named Escalator Boxcar
Train (EBT) \cite{Roos1988}. Rather than approximating the solution
directly, it approximates the measure induced by the solution. Regardless of
its unconventional solution methodology, the EBT method is widely used by
theoretical biologists (see e.g., \cite%
{briggs_dynamical_1995,goetz_using_2008,persson_ontogenetic_1998,xabadia_optimal_2010}%
). One of the reasons for the popularity of the EBT method can be ascribed
to the simple biological interpretation of the components of the scheme: the
state-space is partitioned into initial cohorts and, for the $i$th cohort,
the EBT method tracks its size $N_{i}(t)$ and the location of its centre of
mass $X_{i}(t)$ (see e.g., \cite{Roos1997}). The solution measure $%
d\zeta_{t}:=$ $u(t,x)\,dx$ is then approximated by
\begin{equation}
d\zeta_{t}\approx
d\zeta_{t}^{N}\equiv\sum_{i=B}^{N}N_{i}(t)\delta_{X_{i}(t)},
\end{equation}
where $\delta_{x}$ is the Dirac measure concentrated at $x$. The dynamics of
the functions $N_{i}$ and $X_{i}$ will be defined in Sect. \ref{sec:EBT}.
The boundary cohort corresponding to $i=B$ is treated differentially from
the other cohorts to account for newborn individuals. In the original
formulation, \cite{Roos1988}, this included terms correcting for changes in
the average mass arising from the inflow of newborn individuals. For
completeness, we consider the original definition of the boundary cohort in
Sect. \ref{Section_OD}.

The convergence of the EBT method has remained an open question since the
method was first introduced in 1988. The most successful analysis was
performed by de Roos and Metz \cite{Roos1991} in 1991. They studied how well
the EBT method approximates integrals of the form $\int_{\Omega}%
\psi(x)u(x,t)\,dx$ for smooth functions $\psi,$ assuming that cohorts are
not internalized (see Sect. \ref{sec:EBT}). The result does not assert the 
\emph{convergence} of the EBT method but rather, in the language used by de
Roos and Metz, that the EBT method \emph{consistently} approximates
integrals of the solution to \eqref{PSPM}. One reason for the lack of
progress is that the usual analytical techniques for analyzing finite
element and finite difference schemes are not immediately applicable to the
measure-valued case. Over the last two decades, however, the theory of
structured population models has been extensively developed \cite%
{Diekmann1998,Diekmann2001,Diekmann2005,diekmann_second_2008,diekmann_stability_2008,diekmann_daphnia_2009,Gwiazda2010}%
, and for the first time a full analysis of the EBT method is within our
reach.

The aim of this paper is to rigorously prove the convergence of the EBT\
method. We show that the EBT method converges under far weaker conditions on
the growth, death and birth functions than the conditions assumed by de Roos
and Metz \cite{Roos1991}. Our arguments build on recent theoretical
developments by Gwiazda et al. \cite{Gwiazda2010} that extend the classical
concept of weak solutions to measured-valued solutions. In the following
section, we describe the EBT method in full detail, define weak convergence
of measures, and define weak solutions to the physiologically structured
population model \eqref{PSPM}. In Sect.~\ref{Section_3} we prove the
convergence of the EBT method with dynamics of the boundary cohort as
introduced in this paper. Our convergence result is then extended to the
original definition of the boundary cohort in Sect.~\ref{Section_OD}. We
conclude by placing our results into context and by highlighting promising
directions for future work. Theorem \ref{MainThmWONumerics} and Theorem \ref%
{MainThmEBTConvergence} are the main results of this paper.

\section{The Escalator Boxcar Train}

\label{sec:EBT}

The EBT\ method is a numerical scheme for solving physiologically structured
population models (PSPMs, see e.g., \cite{Metz1986}). While there are many
possible formulations of PSPMs, several of which are described in the
excellent book by Metz and Diekmann \cite{Metz1986}, we consider the
numerical solution of the one-dimensional PSPM with a single birth state $%
x_{b}$ defined by \eqref{transport}, \eqref{transport_boundary}, and %
\eqref{transport_initial}. The EBT\ method determines an approximate
measured-valued solution $\zeta_{t}^{N}$ to the PSPM as a linear combination
of Dirac measures,%
\[
\zeta_{t}^{N}\equiv\sum_{i=B}^{N}N_{i}(t)\delta_{X_{i}(t)}. 
\]
Each of the terms in the approximation can be interpreted biologically as a
cohort composed of $N_{i}$ individuals with average individual state (e.g.,
size) $X_{i}$ at time $t$. As individuals give rise to offspring with state $%
x_{b}$ at birth, we need different definitions for \emph{internal cohorts }%
and the \emph{boundary cohort}.

The internal cohorts are numbered $i=B+1,...,N$. These cohorts are chosen at
time $t=0$ so that $\zeta_{0}^{N}$ converges weakly to the initial data $%
u_{0}(x)dx$ as $N\rightarrow\infty$. This is always possible since finite
linear combinations of Dirac measures are dense in the weak topology \cite[%
Volume II, p. 214]{Bogachev2007}. Thus, we need not restrict ourselves to
initial data prescribed by a function $u_{0}(x),$ but can extend our
analysis to general positive Radon measures $\nu_{0}$. Without loss of
generality, we will assume that the total mass $\zeta_{0}^{N}([x_{b},%
\infty))=\nu_{0}([x_{b},\infty))$ for all $N$. The boundary cohort is the
cohort with the lowest index $B$. At time $t=0$, $B=0$ and we assume that $%
N_{0}(0)=0$ and $X_{0}(0)=x_{b}.$ As time progresses, additional cohorts
with negative index will be created through the process of internalization
described further below.

The dynamics of the internal cohorts are given by 
\begin{subequations}
\label{EBT:internal}
\begin{align}
\frac{dN_{i}}{dt} & =-\mu\left( X_{i},\zeta^{N}\right) N_{i}, \\
\frac{dX_{i}}{dt} & =g\left( X_{i},\zeta^{N}\right) ,
\end{align}
\end{subequations}
where we have assumed a direct dependence of the vital rates on the solution
measure, $\zeta^{N}=\zeta_{t}^{N},$ to represent environmental feedback.
Similarly, but in contrast to the original formulation of the EBT\ method by
de Roos \cite{Roos1988}, the dynamics of the boundary cohorts follow
\begin{subequations}
\label{EBT:boundary}
\begin{align}
\frac{dN_{B}}{dt} & =-\mu(X_{B},\zeta^{N})N_{B}+\sum_{i=B}^{N}\beta
(X_{i},\zeta^{N})N_{i,} \\
\frac{dX_{B}}{dt} & =g(X_{B},\zeta^{N}),
\end{align}
\end{subequations}
where the sum is taken over all cohorts including the boundary cohort. This
sum reflects the offspring produced by the total population. In line with
their biological interpretations, we henceforth assume that all vital rates,
the mortality rate $\mu$, the fecundity rate $\beta$, and the growth rate $g$%
, are non-negative.

With the EBT method defined as above, both the width and the number of
individuals in the boundary cohort will increase over time which eventually
introduces an unacceptably large approximation error. For this reason, the
boundary cohort must be \emph{internalized} sufficiently often. This implies
that the number of cohorts will increase following internalization. The new
boundary cohort is at the time $t$ of the internalization given by $%
N_{B}(t)=0$ and $X_{B}(t)=x_{b}$, where $B$ equals the index of the old
boundary cohort decremented one step. At the same instant, the previous
boundary cohort becomes an internal cohort. To prevent the number of
internal cohorts from exceeding computationally acceptable bounds, internal
cohorts may be removed when the number of individuals has declined
sufficiently. Removal of internal cohorts is important for numerical
implementation but will not be considered in this manuscript.

The EBT method differs from traditional numerical schemes in that it aims to
approximate the solution as a measure of point masses. Before we can discuss
the convergence of the EBT method, it is necessary to extend the classical
concept of a weak solution to measures. This extension builds on earlier
work by Gwiazda et al. \cite{Gwiazda2010} (see also \cite{Carrillo2012, Gwiazda2010b}) and Chapter 8 of the monograph
\cite{Bogachev2007}. We will work with the cone all finite positive Radon
measures denoted $\mathcal{M}_{+}(\Omega)$, where $\Omega$ is a metric space
consisting of all admissible individual states. In our presentation, we
assume $\Omega=[x_{b},\infty]$ and we think of $x\in\Omega$ as the size of
an individual. An important reason for working with finite Radon measures is
that their behavior at infinity is tightly controlled: for each $\epsilon>0$%
, there exists a compact set $K_{\epsilon}$ such that $\mu(\Omega\backslash
K_{\epsilon})<\epsilon$.

Since the EBT method approximates the true solution as a measure of point
masses, the natural mode of convergence on $\mathcal{M}_{+}(\Omega)$ is
\emph{weak convergence}\footnote{%
There are two natural notions of convergence on $\mathcal{M}_{+}(\Omega)$%
---strong convergence and weak convergence. Strong convergence is unsuitable
for our purposes as, for example, the sequence of Dirac measures $%
\delta_{1/n}$ does not converge to $\delta_{0}$ as $n\rightarrow\infty$ in
the strong topology.}:

\begin{definition}
A sequence of measures $\left\{ \mu_{k}\right\} $ on $\Omega$ converges
weakly to a measure $\mu$ if
\[
\int_{\Omega}\phi(x)\,d\mu_{k}(x)\rightarrow\int_{\Omega}\phi(x)\,d\mu(x),
\]
as $k\rightarrow\infty$ for all bounded continuous real functions $\phi$ on $%
\Omega$.
\end{definition}

The weak convergence defined above induces a topology associated with the
Kantorovich-Rubinstein metric:%
\[
\rho(\mu,\nu)=\sup\left\{ \int_{\Omega}\phi(x)\,d(\mu-\nu)\,\Big \vert%
\phi\in C_{0}^{\infty}(\mathbb{R}),\left\Vert \phi\right\Vert
_{W^{1,\infty}}\leq1\right\} , 
\]
in which $\left\Vert \phi\right\Vert _{W^{1,\infty}}=\left\Vert \phi
\right\Vert _{L^{\infty}}+\left\Vert \phi^{\prime}\right\Vert _{L^{\infty}}$%
. This is also known as the flat metric. With this metric, $\mathcal{M}%
_{+}(\Omega)$ is a complete metric space (see \cite[Def. 2.5]{Gwiazda2010}).

Analogously to weak convergence, we define weak continuity as follows:

\begin{definition}
A mapping $\zeta_{t}:\mathbb{R}_{+}\rightarrow\mathcal{M}_{+}(\Omega)$ is
weakly continuous in time if, for all bounded continuous real functions $\phi
$ on $\Omega$, 
\[
\int_{\Omega}\phi(x)d\zeta_{t}, 
\]
is continuous in the classical sense as a function of $t$.
\end{definition}

With these two topological notions in place, we are in position to define
measure-valued solutions to the PSPM\ \eqref{PSPM}:

\begin{definition}
A mapping $\zeta_{t}:[0,T]\rightarrow\mathcal{M}_{+}([0,\infty))$ is a weak
solution to \eqref{PSPM} up to time $T$ if $\zeta_{t}$ is weakly continuous
in time and 
\begin{align}
&
\int_{x_{b}}^{\infty}\phi(x,T)\,d\zeta_{T}(x)-\int_{x_{b}}^{\infty}\phi(x,0)%
\,d\nu_{0}(x)=  \nonumber \\
& \int_{0}^{T}\int_{x_{b}}^{\infty}\left( \frac{\partial\phi}{\partial t}%
(x,t)+g(x,\zeta_{t})\frac{\partial\phi}{\partial x}(x,t)-\mu(x,\zeta
_{t})\phi(x,t)\right) \,d\zeta_{t}(x)dt  \nonumber \\
& +\int_{0}^{T}\phi(x_{b},t)\int_{x_{b}}^{\infty}\beta(x^{\prime},\zeta
_{t})\,d\zeta_{t}(x^{\prime})\,d\zeta_{t}(x)\,dt,   \label{weak_solution}
\end{align}
for all $\phi\in C_{0}^{\infty}(\mathbb{R}_{+}\times\lbrack0,T]).$ Here, $%
\nu_{0}\in\mathcal{M}_{+}(\Omega)$ is the initial data at time $t=0$.
\end{definition}

\begin{remark}
The definition above was inspired by Gwiazda et al. \cite{Gwiazda2010}. We
differ in that we use smooth test functions, but note that these are dense
in the space $C^{1}\cap W^{1,\infty}$ used in \cite{Gwiazda2010}.
\end{remark}

\begin{remark}
The dependence on the environmental feedback variable $E$ in \eqref{PSPM} is
represented here by a direct dependence on the solution measure $\zeta_{t}$.
\end{remark}

In order to show the convergence of the EBT method, we will recast the
definition of a weak solution. Let $0\leq t_{1}<t_{2}\leq T$ and $v\in%
\mathcal{M}_{+}(\Omega)$. For a given test function $\phi\in$ $%
C_{0}^{\infty}(\mathbb{R}_{+}\times\lbrack0,T])$ and a family of measures $%
\sigma_{t}$, we define the residual 
\begin{align}
R_{\phi}(\sigma_{t}, & \nu,t_{1},t_{2})=\int_{x_{b}}^{\infty}\phi
(x,t_{2})\,d\sigma_{t_{2}}(x)-\int_{x_{b}}^{\infty}\phi(x,t_{1})\,d\nu (x)
\label{residual_weak_arbitrary_measure} \\
& -\int_{t_{1}}^{t_{2}}\int_{x_{b}}^{\infty}\left( \frac{\partial\phi }{%
\partial t}(x,t)+g(x,\zeta_{t})\ \dfrac{\partial\phi}{\partial x}%
(x,t)-\mu(x,\zeta_{t})\phi(x,t)\right) \,d\sigma_{t}(x)dt  \nonumber \\
& +\int_{t_{1}}^{t_{2}}\ \phi(x_{b},t)\left(
\int_{x_{b}}^{\infty}\beta(x^{\prime},\zeta_{t})\,d\sigma_{t}(x^{\prime})%
\right) \,dt,  \nonumber
\end{align}
where the measure $\nu$ is interpreted as the initial data at time $t=t_{1}$%
. Clearly, if $R_{\phi}(\sigma_{t},\nu_{0},0,T)=0$ for all test functions $%
\phi$ and the family of measures $\sigma_{t}$ is weakly continuous in time,
then $\sigma_{t}$ is a weak solution to \eqref{PSPM}. We will sometimes
write $R_{\phi}(\sigma_{t})$ meaning $R_{\phi}(\sigma_{t},\nu_{0},0,T)$.

\section{Convergence of the Escalator Boxcar Train\label{Section_3}}

\label{Section_2}

We establish the convergence of the EBT method in five steps: (1) At each
fixed time $t$, the sequence of approximating EBT measures contains a
subsequence which converges weakly to a positive Radon measure $\zeta_{t}$.
(2) We find a subsequence that for all $t$ converges weakly to a mapping $%
\zeta_{t}$ that is weakly continuous in time. (3) The residuals of the
approximating EBT measures $\zeta_{t}^{N}$ converges to the residual of $%
\zeta_{t}$ for any test function. (4) The residual of the approximating EBT
measures $\zeta_{t}^{N}$ converges to zero, and hence the measure $\zeta_{t}$
is a weak solution. All that remains is then to show that the entire
sequence of approximating EBT measures converges weakly to $\zeta_{t}$. We
do this by (5) assuming the existence of a unique weak solution to the
structured population model and showing that a contradiction will otherwise
result. In all the following lemmas, we assume that the birth rate, growth
rate, and mortality rate are non-negative, bounded, and Lipschitz continuous
functions of the individual size $x$. In addition, we need three assumption
pertaining to the feedback from the population-level to individual vital
rates: 
\begin{align*}
\sup_{x}\left\vert \beta(x,\sigma)-\beta(x,\lambda)\right\vert & \leq
C_{\beta}~\rho(\sigma,\lambda), \\
\sup_{x}\left\vert g(x,\sigma)-g(x,\lambda)\right\vert & \leq
C_{g}~\rho(\sigma,\lambda), \\
\sup_{x}\left\vert \mu(x,\sigma)-\mu(x,\lambda)\right\vert & \leq C_{\mu
}~\rho(\sigma,\lambda).
\end{align*}
The three requirements above assert Lipschitz continuity in $\mathcal{M}%
_{+}(\Omega)$ equipped with the Kantorovich-Rubinstein metric.

\begin{lemma}[Step 1]
\label{Lemma_step_1}For each $t\in\lbrack0,T]$, the sequence $%
\{\zeta_{t}^{N}\}$ of approximating EBT measures contains a weakly
convergent subsequence. In fact, any subsequence $\{\zeta_{t}^{N^{\prime}}\}$
of $\{\zeta_{t}^{N}\}$ contains a weakly convergent subsequence.
\end{lemma}

\begin{proof}
By Prohorov's Theorem \cite{Bogachev2007}, it is enough to show that the
sequence $\{\zeta_{t}^{N}\}$ is uniformly bounded in the variation norm and
is uniformly tight. As the measures are positive by construction, this
amounts to showing that $\zeta_{t}^{N}([x_{b},\infty))$ is uniformly bounded
in $N$, with $\lim_{M\rightarrow\infty}\sup_{M}\zeta_{t}^{N}((M,\infty))=0$.
An biological interpretation of these requirements, which we will build on
in the proof, is that the abundance and typical size of individuals in the
population are bounded from above. Letting $P_{N}(s)=\zeta_{s}^{N}([x_{b},%
\infty))$ it follows that%
\begin{align*}
P_{N}^{\prime}(s) &
=\sum_{i=B}^{N}N_{i}^{\prime}(s)=\sum_{i=B}^{N}\beta(X_{i},%
\zeta_{s}^{N})N_{i}(s)-\sum_{i=B}^{N}\mu\left( X_{i},\zeta _{s}^{N}\right)
N_{i}(s)\leq \\
&
\leq\sum_{i=B}^{N}\beta(X_{i},\zeta_{s}^{N})N_{i}(s)\leq\beta_{\sup}%
\sum_{i=B}^{N}N_{i}(s)=\beta_{\sup}P_{N}(s),
\end{align*}
where $\beta_{\sup}$ is the supremum of $\beta$, i.e., the maximum
individual birth rate. The above inequality holds for all $s\in\lbrack0,T]$
except at the finite number of times, where boundary cohorts are
internalized. At these points, the function $P_{N}$ is continuous. Thus, $%
0\leq P_{N}(t)\leq P_{N}(0)\exp(\beta_{\sup}T)$. Hence $P_{N}(t)=%
\zeta_{t}^{N}([x_{b},\infty))$ is uniformly bounded on $[0,T]$, since $%
P_{N}(0)$ is independent of $N$. (Recall that in Sect.~\ref{Section_2} we
assumed that the initial mass should be independent of $N$ and equal to that
of the population measure given as initial condition.)

To prove $\lim_{M\rightarrow\infty}\sup_{N}\zeta_{t}^{N}(~(M,\infty)~)=0$,
we first show that the statement is true for $t=0$. Let $\varepsilon>0$ be
given. Since the initial data $\nu_{0}$ is a positive Radon measure and thus
tightly controlled at infinity, we may choose $M_{1}$ large enough such that 
$\nu _{0}(~(M_{1},\infty)~)<\varepsilon/2.$ Pick any continuos function $%
\varphi$ on $[x_{b},\infty)$ satisfying $0\leq\varphi(x)\leq1$ with $%
\varphi(x)=1$ for $x>M_{1}+1$ and $\varphi(x)=0$ for $x<M_{1}.$ Then 
\[
\zeta_{0}^{N}([M_{1}+1,\infty))\leq\int_{M_{1}}^{\infty}\varphi~d%
\zeta_{0}^{N}<\int_{M_{1}}^{\infty}\varphi~d\nu_{0}+\varepsilon/2<%
\varepsilon, 
\]
if we choose $N>N_{0}$ for some sufficiently large $N_{0},$ since $\zeta
_{0}^{N}$ converges weakly to $\nu_{0}$ as $N\rightarrow\infty.$ To account
for the measures with $N\leq N_{0,}$ we choose $M_{2}$ so large that $%
\zeta_{0}^{N}([M_{2},\infty))<\varepsilon$ for $N=1,2,...,N_{0}$. Finally,
we choose $M\ $as the largest of the two numbers $M_{1}+1$ and $M_{2}$.

To prove the statement for a general time $t\in\lbrack0,T]$, we first note
that the center of mass and abundance at time $t$ of any internal cohort $i>0
$ with $X_{i}(t)$ large enough can be estimated with their respective values
at time $t=0$. Specifically, $X_{i}(t)\leq X_{i}(0)+tg_{\sup},$ where $%
g_{\sup}$ is the supremum of the growth rate $g$, and $N_{i}(t)\leq N_{i}(0).
$ Combining these two estimates, we have that $\zeta_{t}^{N}(M,\infty)\leq%
\zeta_{0}^{N}(M-tg_{\sup},\infty)$ and the first assertion of the lemma
follows. Finally we note that the above argument holds for any subsequence
of $\{\zeta_{t}^{N}\}$. This concludes the proof.
\end{proof}

\begin{lemma}[Step 2]
The approximating EBT sequence $\{\zeta_{t}^{N}\}$ contains a subsequence
which, for each $t\in\lbrack0,T]$, converges weakly to a positive finite
measure $\zeta_{t}$. The mapping $\zeta_{t}:[0,T]\rightarrow \mathcal{M}%
_{+}(\Omega)$ is weakly continuous in time.
\end{lemma}

\begin{proof}
Let $\left\{ q_{k}\right\} _{k=1}^{\infty}$ be an enumeration of the
rational numbers in $[0,T]$. According to Lemma \ref{Lemma_step_1} there
exists a convergent subsequence $\{\zeta_{q_{1}}^{N_{j}^{1}}\}$ of $%
\{\zeta_{q_{1}}^{N}\}$. Repeating this argument, there exists a convergent
subsequence $\{\zeta_{q_{2}}^{N_{j}^{2}}\}$ of $\{\zeta_{q_{2}}^{N_{j}^{1}}\}
$. Proceeding by induction, we obtain for each $k$ a sequence $\{\zeta
_{q_{k}}^{N_{j}^{k}}\}$ which converges weakly to $\zeta_{q_{k}}$ and is a
subsequence of all preceding sequences. Inspired by Cantor's diagonalization
argument we define the sequence $\hat{\zeta}_{t}^{k}:=\zeta_{t}^{N_{k}^{k}}$%
. It follows that for each rational $t\in\lbrack0,T]$, this sequence
converges weakly to a measure $\zeta_{t}$.

We will now show that the subsequence also converges to a positive finite
Radon measure for all real $t\in\lbrack0,T]$. We first show that for each
fixed test function $\phi\in C^{\infty}(\mathbb{R})$ and each time $t$, the
sequence of real numbers 
\begin{equation}
\int_{x_{b}}^{\infty}\phi~d\hat{\zeta}_{t}^{k}, 
\label{convergence_in_distributtions}
\end{equation}
converges as $k\rightarrow\infty$. It then follows from classical results in
the theory of distributions, e.g., \cite[Theorem 2.1.8 and Theorem 2.1.9]%
{Hormander_1990_book}, that $\hat{\zeta}_{t}^{k}$ converges weakly to a
positive measure $\zeta_{t}$. This will turn out to be the desired measure.

To prove convergence of the sequence (\ref{convergence_in_distributtions}),
we first note that for fixed $k$, the measure $\hat{\zeta}_{t}^{k}$ is
weakly continuous in time since each $N_{i}(.)$ and $X_{i}(.)$ are
continuous functions. Let $t\in\lbrack0,T]$ and $\phi$ be a test function.
Given $\varepsilon>0$ we get%
\begin{gather*}
\left\vert \int_{x_{b}}^{\infty}\phi~d\hat{\zeta}_{t}^{j}-\int_{x_{b}}^{%
\infty}\phi~d\hat{\zeta}_{t}^{k}\right\vert \leq \\
\leq\left\vert \int_{x_{b}}^{\infty}\phi~d\hat{\zeta}_{t}^{j}-\int_{x_{b}}^{%
\infty}\phi~d\hat{\zeta}_{q}^{j}\right\vert +\left\vert
\int_{x_{b}}^{\infty}\phi~d\hat{\zeta}_{q}^{j}-\int_{x_{b}}^{\infty}\phi~d%
\hat{\zeta}_{q}^{k}\right\vert +\left\vert \int_{x_{b}}^{\infty}\phi~d\hat{%
\zeta}_{q}^{k}-\int_{x_{b}}^{\infty}\phi~d\hat{\zeta}_{t}^{k}\right\vert ,
\end{gather*}
for any $j$, $k$, and $q$. Noting that the birth rate and mortality rate are
bounded, we can use the same argument as in the proof of Lemma \ref%
{Lemma_step_1} to show that the first and last term above are bounded by a
constant multiple of $\left\vert t-q\right\vert $. In particular, this
constant depends on neither $j$ nor $k$. Choosing $q$ as a rational number
sufficiently close to $t$ these two terms will be smaller than $\varepsilon
/2$. Finally, since $q$ is rational, we may choose $j$ and $k$ large enough
to make the middle term less than $\varepsilon/2$. Thus, we have established
the Cauchy property for the sequence (\ref{convergence_in_distributtions}),
which hence converges for all test functions $\phi$. This shows that $\hat{%
\zeta }_{t}^{j}$ converges weakly to a bounded positive Radon measure $\hat{%
\zeta }_{t}$ for all $t\in\lbrack0,T]$.

Using the same idea as above, we see that $\hat{\zeta}_{t}$ is weakly
continuous in time. Specifically,%
\begin{gather*}
\left\vert \int_{x_{b}}^{\infty}\phi~d\hat{\zeta}_{s}-\int_{x_{b}}^{\infty
}\phi~d\hat{\zeta}_{t}\right\vert \leq \\
\leq\left\vert \int_{x_{b}}^{\infty}\phi~d\hat{\zeta}_{s}-\int_{x_{b}}^{%
\infty}\phi~d\hat{\zeta}_{s}^{k}\right\vert +\left\vert
\int_{x_{b}}^{\infty}\phi~d\hat{\zeta}_{s}^{k}-\int_{x_{b}}^{\infty}\phi~d%
\hat{\zeta}_{t}^{k}\right\vert +\left\vert \int_{x_{b}}^{\infty}\phi~d\hat{%
\zeta}_{t}^{k}-\int_{x_{b}}^{\infty}\phi~d\hat{\zeta}_{t}\right\vert ,
\end{gather*}
where again the middle term is bounded by a constant multiple of $\left\vert
t-s\right\vert $ independent of $k$. Finally, the first and last term can be
made arbitrarily small as a consequence of the weak convergence of $\hat {%
\zeta}_{s}^{k}$ to $\hat{\zeta}_{s}$.
\end{proof}

\begin{lemma}
\label{continuous_function_integral}Assume that the sequence $\zeta_{t}^{k}$
converges weakly to a finite Radon measure $\zeta_{t}$. If $\varphi\in
C_{0}^{\infty}(\mathbb{R}_{+}\times\lbrack0,T])$ then, for every bounded
Lipschitz continuous function $f$ satisfying 
\[
\sup_{x}\left\vert f(x,\sigma)-f(x,\lambda)\right\vert \leq C_{f}~\rho
(\sigma,\lambda), 
\]
for all $\sigma,\lambda\in\mathcal{M}_{+}(\Omega)$, we get%
\[
\int_{0}^{T}\int_{x_{b}}^{\infty}\varphi(x,t)f(x,\zeta_{t}^{k})\,d\zeta
_{t}^{k}(x)dt\rightarrow\int_{0}^{T}\int_{x_{b}}^{\infty}\varphi
(x,t)f(x,\zeta_{t})\,d\zeta_{t}(x)dt, 
\]
as $k$ tends to infinity.
\end{lemma}

\begin{proof}
We have%
\begin{align}
\int_{x_{b}}^{\infty}\varphi(x,t)f(x,\zeta_{t}^{k})\,d\zeta_{t}^{k}(x) &
=\int_{x_{b}}^{\infty}\varphi(x,t)f(x,\zeta_{t})\,d\zeta_{t}^{k}(x)+
\label{integration_of_test_functions2} \\
& +\int_{x_{b}}^{\infty}\varphi(x,t)\left( f(x,\zeta_{t}^{k})-f(x,\zeta
_{t})\right) \,d\zeta_{t}^{k}(x).  \nonumber
\end{align}
In the first term on the right hand side, the function $\varphi(x,t)f(x,%
\zeta _{t})$ is bounded and Lipschitz continuous in $x.$ Hence, it can be
approximated by a sequence $\{\varphi_{m}\}$ of functions in $C_{0}^{\infty
}(\mathbb{R}_{+}\times\lbrack0,T])$ that converges pointwise and in $%
W^{1,\infty}$-norm. Such a sequence can, for example, be constructed through
convolution. As the first term would vanish if $\left\Vert \varphi
(\cdot,t)f(\cdot,\zeta_{t})\right\Vert _{W^{1,\infty}}=0$, we can assume
that this is not the case. We then get%
\begin{align*}
\int_{x_{b}}^{\infty}\varphi(x,t)f(x,\zeta_{t})\,d\zeta_{t}^{k}(x) &
=\int_{x_{b}}^{\infty}\lim_{m\rightarrow\infty}\varphi_{m}(x,t)\,d\zeta
_{t}^{k}(x)= \\
& =\lim_{m\rightarrow\infty}\int_{x_{b}}^{\infty}\varphi_{m}(x,t)\,d\zeta
_{t}^{k}(x),
\end{align*}
where we have used Lebesgue's dominated convergence theorem. Hence,%
\begin{gather*}
\left\vert
\int_{x_{b}}^{\infty}\varphi(x,t)f(x,\zeta_{t})\,d\zeta_{t}^{k}(x)-%
\int_{x_{b}}^{\infty}\varphi(x,t)f(x,\zeta_{t})\,d\zeta_{t}(x)\right\vert =
\\
=\lim_{m\rightarrow\infty}\left\vert \int_{x_{b}}^{\infty}\varphi
_{m}(x,t)\,d(\zeta_{t}^{k}-\zeta_{t})(x)\right\vert = \\
\leq\lim_{m\rightarrow\infty}\left\Vert \varphi_{m}(.,t)\right\Vert
_{W^{1,\infty}}\rho(\zeta_{t}^{k}-\zeta_{t})= \\
=\left\Vert \varphi(.,t)f(.,\zeta_{t})\right\Vert _{W^{1,\infty}}\rho
(\zeta_{t}^{k}-\zeta_{t})\rightarrow0,
\end{gather*}
as $k$ tends to infinity. Thus the first term converges to 
\[
\int_{x_{b}}^{\infty}\varphi(x,t)f(x,\zeta_{t})\,d\zeta_{t}(x). 
\]
It remains to show that the second term in %
\eqref{integration_of_test_functions2} vanishes as $k\rightarrow\infty$,%
\begin{gather*}
\left\vert \int_{x_{b}}^{\infty}\varphi(x,t)\left(
f(x,\zeta_{t}^{k})-f(x,\zeta_{t})\right) \,d\zeta_{t}^{k}(x)\right\vert \leq
\\
\leq\sup_{x}\left\vert \varphi(x,t)\left( f(x,\zeta_{t}^{k})-f(x,\zeta
_{t})\right) \,\right\vert \zeta_{t}^{k}([x_{b},\infty))\leq \\
\leq\sup_{x}\left\vert \varphi(x,t)\right\vert \sup_{x}\left\vert
f(x,\zeta_{t}^{k})-f(x,\zeta_{t})\right\vert
\,\zeta_{t}^{k}([x_{b},\infty))\leq \\
\leq
C_{\varphi}C_{f}~\rho(\zeta_{t}^{k},\zeta_{t})\,\zeta_{t}^{k}([x_{b},%
\infty)).
\end{gather*}
Since $\zeta_{t}^{k}$ converges weakly to $\zeta_{t},$ it follows from
Gwiazda et al. \cite[Theorem 2.7]{Gwiazda2010} that $\,\zeta_{t}^{k}([x_{b},%
\infty))$ is uniformly bounded and $\rho(\zeta_{t}^{k},\zeta_{t})$ tends to
zero as $k$ tends to infinity. Since the above calculation is done pointwise
in $t$, the lemma follows from Lebesgue's dominated convergence theorem.
\end{proof}

\begin{lemma}[Step 3]
\label{Step3}Assume that the sequence $\zeta_{t}^{k}$ converges weakly to
the finite Radon measure $\zeta_{t}$. Then the residual $R_{\phi}(\zeta
_{t}^{k})$ converges to $R_{\phi}(\zeta_{t})$ for all test functions $%
\phi\in C_{0}^{\infty}(\mathbb{R}_{+}\times\lbrack0,T])$.
\end{lemma}

\begin{proof}
Consider%
\begin{align}
R_{\phi}(\zeta_{t}^{k})= &
\int_{x_{b}}^{\infty}\phi(x,T)\,d\zeta_{t}^{k}(x)-\int_{x_{b}}^{\infty}%
\phi(x,0)\,d\nu_{0}(x) \\
& -\int_{0}^{T}\int_{x_{b}}^{\infty}\left( \frac{\partial\phi}{\partial t}%
(x,t)+g(x,\zeta_{t}^{k})\ \frac{\partial\phi}{\partial x}(x,t)-\mu
(x,\zeta_{t}^{k})\phi(t,x)\right) \,d\zeta_{t}^{k}(x)dt  \nonumber \\
+ & \int_{0}^{T}\phi(x_{b},t)\left( \int_{x_{b}}^{\infty}\beta(x^{\prime
},\zeta_{t}^{k})\,d\zeta_{t}^{k}(x^{\prime})\right) \,dt=I-II-III+IV. 
\nonumber
\end{align}
The first term converges by definition of weak convergence and the second
term is unchanged. The third and fourth term converge by Lemma \ref%
{continuous_function_integral}.
\end{proof}

\begin{lemma}
\label{Residual_without_internalization}Let $0\leq t_{1}<t_{2}\leq T$ and $%
v\in\mathcal{M}_{+}(\Omega)$. Assuming that no internalization is done in
the interval $(t_{1},t_{2})$, then for any test function $\phi$ we have that 
\begin{align*}
R_{\phi}(\zeta_{t}^{N},\nu,t_{1},t_{2}) &
=\sum_{i=B}^{N}N_{i}(t_{1})\phi(X_{i}(t_{1}),t_{1})-\int_{x_{b}}^{\infty}%
\phi(x,t_{1})\,d\nu(x)+ \\
&
+\int_{t_{1}}^{t_{2}}(\phi(X_{B}(t),t)-\phi(x_{b},t))\sum_{i=B}^{N}%
\beta(X_{i}(t),\zeta_{t}^{N})\ N_{i}(t)dt,
\end{align*}
where the sum is taken over all cohorts, including the boundary cohort.
\end{lemma}

\begin{proof}
We write the residual (\ref{residual_weak_arbitrary_measure}) as%
\begin{align*}
R_{\phi}(\zeta_{t}^{N},\nu,t_{1},t_{2}) & =\int_{x_{b}}^{\infty}\phi
(x,t_{2})\,d\zeta_{t_{2}}^{N}(x)-\int_{x_{b}}^{\infty}\phi(x,t_{1})\,d\nu(x)
\\
& -\int_{t_{1}}^{t_{2}}\int_{x_{b}}^{\infty}\left( \phi_{2}(\xi
,t)+g(x,\zeta_{t}^{N})\ \phi_{1}(x,t)-\mu(x,\zeta_{t}^{N})\phi(x,t)\right)
\,d\zeta_{t}^{N}(x)dt \\
& +\int_{t_{1}}^{t_{2}}\ \phi(x_{b},t)\left(
\int_{x_{b}}^{\infty}\beta(x^{\prime},\zeta_{t}^{N})\,d\zeta_{t}^{N}(x^{%
\prime})\right) \,dt= \\
& =I(\zeta_{t_{2}}^{N})-II(\nu)-III(\zeta_{t}^{N})-IV(\zeta_{t}^{N}).
\end{align*}
Here we have used the shorthand notation $\phi_{1}(\xi,t)=\partial\phi
(\xi,t)/\partial x$ and $\phi_{2}(\xi,t)=\partial\phi(\xi,t)/\partial t$.

Recalling that 
\[
\zeta_{t}^{N}=\sum_{i=B}^{N}N_{i}(t)\delta_{X_{i}(t)}, 
\]
we get 
\[
I(\zeta_{t_{2}}^{N})=\sum_{i=B}^{N}N_{i}(t_{2})\phi(X_{i}(t_{2}),t_{2}), 
\]%
\begin{gather*}
III(\zeta_{t}^{N})=\sum_{i=B}^{N}\int_{t_{1}}^{t_{2}}N_{i}(t)\text{ }\left(
\phi_{2}(X_{i}(t),t)+g(X_{i}(t), \zeta_{t}^{N})\ \phi_{1}(X_{i}(t),t)-\mu
(x_{i}(t),\zeta_{t}^{N})\phi(X_{i}(t),t)\right) dt= \\
=III_{B}(\zeta_{t}^{N})+\sum_{i=B+1}^{N}III_{i}(\zeta_{t}^{N}).
\end{gather*}
Now, by (\ref{EBT:internal}), we have%
\begin{align*}
III_{i}(\zeta_{t}^{N}) & =\int_{t_{1}}^{t_{2}}N_{i}(t)\text{ }%
\phi_{2}(X_{i}(t),t)+N_{i}(t)\text{ }\frac{dX_{i}(t)}{dt}\
\phi_{1}(x_{i}(t),t)+\frac{dN_{i}(t)}{dt}\phi(X_{i}(t),t)dt= \\
& =\int_{t_{1}}^{t_{2}}\frac{d}{dt}\left( N_{i}(t)\phi(X_{i}(t),t)\right)
dt=N_{i}(t_{2})\phi(X_{i}(t_{2}),t_{2})-N_{i}(t_{1})\phi(X_{i}(t_{1}),t_{1}).
\end{align*}
Thus 
\[
I(\zeta_{t_{2}}^{N})-\sum_{i=B+1}^{N}III_{i}(\zeta_{t}^{N})=N_{B}(t_{2})%
\phi(X_{B}(t_{2}),t_{2})+\sum_{i=B+1}^{N}N_{i}(t_{1})%
\phi(X_{i}(t_{1}),t_{1}). 
\]
In the same way, but now also using (\ref{EBT:boundary}), we get%
\begin{align*}
III_{B}(\zeta_{t}^{N}) & =\int_{t_{1}}^{t_{2}}\frac{d}{dt}\left( N_{B}(t)%
\text{ }\phi(X_{B}(t),t)\right)
-\phi(X_{B}(t),t)\sum_{i=B}^{N}\beta(X_{i}(t),\zeta_{t}^{N})\ N_{i}(t)dt= \\
& =N_{B}(t_{2})\text{ }\phi(X_{B}(t_{2}),t_{2})-N_{B}(t_{1})\text{ }%
\phi(X_{B}(t_{1}),t_{1})- \\
& -\int_{t_{1}}^{t_{2}}\phi(X_{B}(t),t)\sum_{i=B}^{N}\beta(X_{i}(t),\zeta
_{t}^{N})\ N_{i}(t)dt.
\end{align*}
Since 
\[
IV(\zeta_{t}^{N})=\int_{t_{1}}^{t_{2}}\phi(x_{b},t)\sum_{i=B}^{N}\beta
(X_{i}(t),\zeta_{t}^{N})\ N_{i}(t)dt, 
\]
we have%
\begin{align*}
-III_{B}(\zeta_{t}^{N})-IV(\zeta_{t}^{N}) & =-N_{B}(t_{2})\text{ }\phi
(X_{B}(t_{2}),t_{2})+N_{B}(t_{1})\text{ }\phi(X_{b}(t_{1}),t_{1})+ \\
&
+\int_{t_{1}}^{t_{2}}(\phi(X_{B}(t),t)-\phi(x_{b},t))\sum_{i=B}^{N}%
\beta(X_{i}(t),\zeta_{t}^{N})\ N_{i}(t)dt.
\end{align*}
Summing up the calculations above, we get%
\begin{gather*}
I(\zeta_{t_{2}}^{N})-\sum_{i=B+1}^{N}III_{i}(\zeta_{t}^{N})-III_{B}(\zeta
_{t}^{N})-IV(\zeta_{t}^{N})= \\
=N_{B}(t_{2})\phi(X_{B}(t_{2}),t_{2})+\sum_{i=B+1}^{N}N_{i}(t_{1})\phi
(X_{i}(t_{1}),t_{1})-N_{B}(t_{2})\text{ }\phi(X_{B}(t_{2}),t_{2})+ \\
+N_{B}(t_{1})\text{ }\phi(X_{B}(t_{1}),t_{1}) +\int_{t_{1}}^{t_{2}}(\phi
(X_{B}(t),t)-\phi(x_{b},t))\sum_{i=B}^{N}\beta(X_{i}(t),\zeta_{t}^{N})\
N_{i}(t)dt= \\
=\sum_{i=B}^{N}N_{i}(t_{1})\phi(X_{i}(t_{1}),t_{1})+\int_{t_{1}}^{t_{2}}(%
\phi(X_{B}(t),t)-\phi(x_{b},t))\sum_{i=B}^{N}\beta(X_{i}(t),\zeta_{t}^{N})\
N_{i}(t)dt.
\end{gather*}
Finally we get that 
\begin{align*}
R_{\phi}(\zeta_{t}^{N},\nu,t_{1},t_{2}) &
=\sum_{i=B}^{N}N_{i}(t_{1})\phi(X_{i}(t_{1}),t_{1})-\int_{x_{b}}^{\infty}%
\phi(x,t_{1})\,d\nu(x)+ \\
&
+\int_{t_{1}}^{t_{2}}(\phi(X_{B}(t),t)-\phi(x_{b},t))\sum_{i=B}^{N}%
\beta(X_{i}(t),\zeta_{t}^{N})\ N_{i}(t)dt.
\end{align*}
\end{proof}

\begin{remark}
The residual can be interpreted as the sum of the error arising from the
discretization of the initial data and the error arising from the boundary
cohort. In the interior of the individual state space, the EBT\ method gives
an exact solution, i.e., there are no errors arising from the transportation
of the interior cohorts.
\end{remark}

\begin{lemma}[Step 4]
\label{Lemma_step_4}With $\zeta_{t}^{N}$ defined by the EBT method with
internalizations at times $t_{i}=iT/n$, we have that%
\[
R_{\phi}(\zeta_{t}^{N},\nu_{0},0,T)\rightarrow0, 
\]
as $N$ and $n$ tends to infinity. Here $\nu_{0}$ is the initial data at time 
$t=t_{0}=0$.
\end{lemma}

\begin{proof}
We first write 
\[
R_{\phi}(\zeta_{t}^{N},\nu_{0},0,T)=R_{\phi}(\zeta_{t}^{N},\nu_{0},0,t_{1})+%
\sum_{i=1}^{n-1}R_{\phi}(\zeta_{t}^{N},\zeta_{t_{i}}^{N},t_{i},t_{i+1}). 
\]
By Lemma \ref{Residual_without_internalization} we have,%
\begin{align*}
R_{\phi}(\zeta_{t}^{N},\nu_{0},0,t_{1}) & =\sum_{i=B}^{N}N_{i}(0)\phi
(x_{i}(0),0)-\int_{x_{b}}^{\infty}\phi(x,0)\,d\nu_{0}(x)+ \\
& +\int_{0}^{t_{1}}(\phi(X_{B}(t),t)-\phi(x_{b},t))\sum_{i=B}^{N}\beta
(X_{i}(t),\zeta_{t}^{N})\ N_{i}(t)dt,
\end{align*}
and%
\[
R_{\phi}(\zeta_{t}^{N},\zeta_{t_{i}}^{N},t_{i},t_{i+1})=%
\int_{t_{i}}^{t_{i+1}}(\phi(X_{B}(t),t)-\phi(x_{b},t))\sum_{j=B}^{N}%
\beta(x_{j}(t),\zeta_{t}^{N})\ N_{j}(t)dt. 
\]
A straightforward estimate now gives%
\begin{align*}
\left\vert R_{\phi}(\zeta_{t}^{N},\nu_{0},0,T)\right\vert & \leq\left\vert
\sum_{i=B}^{N}N_{i}(0)\phi(x_{i}(0),0)-\int_{x_{b}}^{\infty}\phi
(x,0)\,d\nu_{0}(\xi)\right\vert + \\
& +\sum_{i=0}^{n-1}\int_{t_{i}}^{t_{i+1}}\left\vert \phi(X_{B}(t),t)-\phi
(x_{b},t)\right\vert \sum_{j=B}^{N}\beta(x_{j}(t),\zeta_{t}^{N})\ N_{j}(t)dt.
\end{align*}
The first term tends to zero by assumption as the number of initial cohorts, 
$N,$ tends to infinity. Noting that $x_{b}=X_{B}(t_{1})$ and using that the
growth rate is bounded, we get 
\[
\left\vert \phi(X_{B}(t),t)-\phi(x_{b},t)\right\vert \leq C_{\phi}\left\vert
X_{B}(t)-x_{b}\right\vert \leq C_{\phi g}\left\vert t-t_{i}\right\vert . 
\]
Hence, 
\begin{gather*}
\sum_{i=0}^{n-1}\int_{t_{i}}^{t_{i+1}}\left\vert \phi(X_{B}(t),t)-\phi
(x_{b},t)\right\vert \sum_{j=B}^{N}\beta(x_{j}(t),\zeta_{t}^{N})\
N_{j}(t)dt\leq \\
\leq\sum_{i=0}^{n-1}C_{\phi g}\left\vert t_{i+1}-t_{i}\right\vert
^{2}C_{\beta\nu_{0}},
\end{gather*}
for the constant $C_{\beta\nu_{0}}=\beta_{\sup}\nu_{0}([x_{b},\infty
))\exp(\beta_{\sup}T)$. Thus, the last sum is bounded by $C(T)/n$ which also
tends to zero as the number of internalizations tends to infinity.
\end{proof}

\begin{remark}
Examining the proof above, we see that the residual tends to zero whenever
the maximal time between two internalizations of the boundary cohort tends
to zero. Hence, we can relax the assumption that the times at which the
boundary cohort is internalized are evenly distributed.
\end{remark}

Recalling that the initial cohorts are chosen to converge weakly to the
initial data, we are now able to prove convergence of the Escalator Boxcar
Train:

\begin{theorem}
\label{MainThmWONumerics}Assume that the assumptions on the birth, growth,
and mortality rates in the beginning of Sect. \ref{Section_3} hold. If the
structured population model given by \eqref{transport}, %
\eqref{transport_boundary}, and \eqref{transport_initial} has a unique
solution $\zeta_{t}$, then the the solutions $\zeta_{t}^{N}$ given by the
EBT method converge weakly to $\zeta_{t}$ as the number of initial cohorts
tends to infinity and the maximal time between two boundary cohort
internalizations tends to zero.
\end{theorem}

\begin{proof}
(Step 5) We assume that the entire sequence $\zeta_{t}^{N}$ does not
converge to $\zeta_{t}$. Then, in the weak topology, there exists an open
neighborhood $U$ of $\zeta_{t}$, and a subsequence $\zeta_{t}^{N_{k}}$ of $%
\zeta_{t}^{N}$ such that $\zeta_{t}^{N_{k}}\notin U$ for all $N_{k}$. From
Lemma \ref{Lemma_step_1}-\ref{Lemma_step_4}, we conclude that $%
\{\zeta_{t}^{N_{k}}\}$ contains a convergent sub-sequence with a limit point
not equal to $\zeta_{t}$, which is a contradiction since it would imply that
the solution to the PSPM is not unique.
\end{proof}

The proof of convergence assumed exact solutions to the ordinary
differential equations (ODEs) underlying the EBT\ method. In practical
implementations, these need to be solved numerically which introduces small
but finite approximation errors. We now extend the convergence proof to
account for errors introduced by the underlying ODE\ solver.

The following lemma is an immediate consequence of Lemma \ref{Step3}.

\begin{lemma}
\label{numerical_approximation}Assume that $\zeta_{t}^{N,h}=%
\sum_{i=B}^{N}N_{i}^{h}(t)\delta_{X_{i}^{h}(t)}$. If for each $t$ we have
that $N_{i}^{h}(t)\rightarrow N_{i}(t)$ and $X_{i}^{h}(t)\rightarrow X_{i}(t)
$ as $h\searrow0$ then $R_{\phi}(\zeta_{t}^{N,h},\nu_{0},0,T)\rightarrow
R_{\phi }(\zeta_{t}^{N},\nu_{0},0,T)$ as $h\searrow0$.
\end{lemma}

Combining the lemma above with Theorem \ref{MainThmWONumerics} we finally
have

\begin{theorem}
\label{MainThmEBTConvergence}Assume that the assumptions on the birth,
growth, and mortality rates in the beginning of Sect. \ref{Section_3} hold.
If the structured population model given by \eqref{transport}, %
\eqref{transport_boundary}, and \eqref{transport_initial} has a unique
solution $\zeta_{t}$, then the the solutions $\zeta_{t}^{N,h}$, given by the
numerical integration of the EBT method, converges weakly to $\zeta_{t}$ if
the number of initial cohorts tends to infinity and the maximal time between
two boundary cohort internalizations tends to zero, while $h$ tends to zero
sufficiently fast.
\end{theorem}

\section{The original definition of the boundary cohort}

\label{Section_OD}

Our study of convergence of the Escalator Boxcar Train in Sect.~\ref%
{Section_3} assumed different dynamics of the boundary cohorts than was used
in the original formulation of the method by de Roos \cite{Roos1988}. We
based our work on the assumption that the boundary cohort differed from the
interior cohorts only in the addition of a term for the inflow of newborns.
In this section, we consider the convergence of the EBT method under the
original definition of the boundary cohort dynamics.

While we simply assumed a dynamical system for the boundary cohort, de Roos
formally derived the underlying equations. Consequently, the original
dynamics for the boundary cohort reflect the reduction in center of mass
that in reality accompanies an inflow of newborns. Moreover, as the center
of mass is not defined as a physical quantity for an empty cohort, the
equations were derived through series expansion around the size at birth.
Thus, rather than tracking the center of mass $X_{B}(t)$ directly, de Roos
considered a quantity $\pi_{B}$ which roughly represents the cumulative
amount by which the individuals in the boundary cohort exceed their birth
size. This quantity is mapped onto the center of mass through the non-linear
transformation%
\begin{equation}
X_{B}=\left\{ 
\begin{array}{cc}
\dfrac{\pi_{B}}{N_{B}}+x_{b}, & \text{if }\pi_{B}>0, \\ 
x_{b}, & \text{otherwise.}%
\end{array}
\right.   \label{nonlinear_transformation}
\end{equation}
The specific equations used for defining the boundary cohort were%
\begin{align}
\dfrac{dN_{B}}{dt} & =-\mu(x_{b},\zeta^{N})N_{B}-\dfrac{\partial\mu
(x_{b},\zeta^{N})}{\partial x}\pi_{B}+\sum_{i=B}^{N}\beta(X_{i},\zeta
^{N})N_{i},  \label{N_deRoos_definition} \\
\dfrac{d\pi_{B}}{dt} & =g(x_{b},\zeta^{N})N_{B}+\dfrac{\partial
g(x_{b},\zeta^{N})}{\partial x}\pi_{B}-\mu(x_{b},\zeta^{N})\pi_{B}, 
\label{pi_deRoos_definition}
\end{align}
with initial conditions $N_{B}=\pi_{B}=0$. We will assume that these are
non-negative, as this is a natural requirement which can easily be enforced
by an ODE\ solver if necessary. The appearance of partial derivatives in the
expressions above, arising from series expansion around the size at birth,
in conjunction with the non-linear transformation mapping $\pi_{B}$ and $%
N_{B}$ onto $X_{B}$, pose new challenges for proving convergence. As we will
show, however, our proof of convergence can be tailored to accompany also
the original definition of the boundary cohort.

Note first that the only parts in the proof of convergence in which the
equations defining the boundary cohort are used is Lemma \ref{Step3} and
implicitly in Theorem \ref{MainThmEBTConvergence}. It therefore suffices to
give new proofs of these statements. To this end, we require an additional
lemma concerning the behavior of the quotient $\pi_{B}/N_{B}$:

\begin{lemma}
With $N_{B}$ and $\pi_{B}$ defined by \eqref{N_deRoos_definition} and %
\eqref{pi_deRoos_definition}, we get 
\[
0\leq X_{B}-x_{b}\leq Ct, 
\]
for $t\in\lbrack t_{0},t_{0}+h]$ and some positive constants $C$ and $h$
which only depend on\ $g$, $\partial g/\partial x$ and $\partial\mu/\partial
x$.
\end{lemma}

\begin{proof}
From the definitions of $N_{B}$ and $\pi_{B}$ we have%
\begin{align*}
\dfrac{d}{dt}X_{b} & =\dfrac{d}{dt}\dfrac{\pi_{B}}{N_{B}}=\dfrac{1}{N_{B}}%
\dfrac{d\pi_{B}}{dt}-\dfrac{\pi_{B}}{N_{B}^{2}}\dfrac{dN_{B}}{dt}= \\
& =g+\dfrac{\partial g}{\partial x}\dfrac{\pi_{B}}{N_{B}}-\mu\dfrac{\pi_{B}}{%
N_{B}}+\mu\dfrac{\pi_{B}}{N_{B}}+\dfrac{\partial\mu}{\partial x}\dfrac {%
\pi_{B}^{2}}{N_{B}^{2}}-\dfrac{\pi_{B}}{N_{B}^{2}}\sum_{i=B}^{N}\beta
_{i}N_{i}= \\
& =g+\dfrac{\partial g}{\partial x}\dfrac{\pi_{B}}{N_{B}}+\dfrac{\partial\mu 
}{\partial x}\dfrac{\pi_{B}^{2}}{N_{B}^{2}}-\dfrac{\pi_{B}}{N_{B}^{2}}%
\sum_{i=B}^{N}\beta_{i}N_{i}\leq \\
& \leq g+\dfrac{\partial g}{\partial x}\dfrac{\pi_{B}}{N_{B}}+\dfrac {%
\partial\mu}{\partial x}\left( \dfrac{\pi_{B}}{N_{B}}\right) ^{2}.
\end{align*}
Remembering that $X_{B}=\pi_{B}/N_{B}+x_{b}$, we thus have $%
X_{B}^{\prime}\leq a+b(X_{B}-x_{b})+c(X_{B}-x_{b})^{2}$ for some positive
constants $a,b$ and $c$. Hence $X_{B}^{\prime}\leq2a$ when $X_{B}\leq
X_{B}^{\ast}$ for some positive $X_{B}^{\ast}$. Since $X_{B}(0)=x_{b}$ it
follows that $X_{B}(t)\leq x_{b}+2at$ for $t\in\lbrack0,X_{B}^{\ast}/2a].$
\end{proof}

We now use this to show that

\begin{lemma}
\label{bounded_terms_original_def}Assume that a new boundary cohort is
created at time $t=t_{1}$. For $t_{2}>t_{1}$ sufficiently close to $t_{1}$,
we have for all $t\in\lbrack t_{1},t_{2}]$ that 
\begin{equation}
\left\vert N_{B}(t)\left( \dfrac{dX_{B}(t)}{dt}-g(X_{B}(t),\zeta_{t}^{N})%
\right) \right\vert \leq C_{1}(t_{2}-t_{1}),   \label{NB_growth}
\end{equation}
and 
\begin{equation}
\left\vert \pi_{B}(t)\dfrac{\partial}{\partial x}\mu(X_{B}(t),\zeta_{t}^{N})%
\right\vert \leq C_{2}(t_{2}-t_{1}).   \label{Pi_growth_rate}
\end{equation}
\end{lemma}

\begin{proof}
Since $N_{B}$ is bounded, $\pi_{B}=N_{B}(X_{B}-x_{b})$, it follows from the
above proof that $|\pi_{B}(t)|\leq C(t_{2}-t_{1})$ for some positive
constant $C$. Hence, since also $\partial\mu(X_{B}(t),\zeta_{t}^{N})/%
\partial x$ is bounded by the assumptions in \cite{Roos1988}, the statement (%
\ref{Pi_growth_rate}) follows trivially. To show the first part of the
assertion, we note that 
\[
N_{B}(t)\left( \dfrac{dX_{B}(t)}{dt}-g(X_{B}(t),\zeta_{t}^{N})\ \right) =%
\dfrac{\partial g}{\partial x}\pi_{B}+\dfrac{\partial\mu}{\partial x}\pi _{B}%
\dfrac{\pi_{B}}{N_{B}}-\dfrac{\pi_{B}}{N_{B}}\sum_{i=B}^{N}\beta_{i}N_{i}. 
\]
Since $\pi_{B}$ and $\pi_{B}/N_{B}=X_{B}-x_{b}$ both increases at most
linearly from zero, the assertion (\ref{NB_growth}) follows.
\end{proof}

The two lemmas above will be used to bound the residual between two
internalizations.

\begin{lemma}
\label{Residual_without_internalization_original_definition}Let $0\leq
t_{1}<t_{2}\leq T$ and $v\in\mathcal{M}_{+}(\Omega)$. For a given test
function $\phi\in C_{0}^{\infty}(\mathbb{R}_{+}\times\lbrack0,T])$ and a
family of measures $\sigma_{t}$. Assuming that no internalization is done in
the interval $(t_{1},t_{2})$, then 
\begin{align*}
R_{\phi}(\zeta_{t}^{N},\nu,t_{1},t_{2}) &
=\sum_{i=B}^{N}N_{i}(t_{1})\phi(X_{i}(t_{1}),t_{1})-\int_{x_{b}}^{\infty}%
\phi(x,t_{1})\,d\nu(x)+ \\
&
+\int_{t_{1}}^{t_{2}}(\phi(X_{B}(t),t)-\phi(x_{b},t))\sum_{i=B}^{N}%
\beta(X_{i}(t),\zeta_{t}^{N})\ N_{i}(t)dt+ \\
& +\int_{t_{1}}^{t_{2}}N_{B}(t)\left( \dfrac{dX_{B}(t)}{dt}%
-g(X_{B}(t),\zeta_{t}^{N})\ \right) \text{ }\phi_{1}(X_{B}(t),t)+ \\
& +\left( \mu_{1}(X_{B}(t),\zeta_{t}^{N})\pi_{B}(t)\right) \phi
(X_{B}(t),t)dt,
\end{align*}
where the sums are taken over all cohorts, including the boundary cohort.
\end{lemma}

\begin{proof}
Examining the proof of Lemma \ref{Residual_without_internalization} we see
that the boundary cohort only appears in the term $III_{B}$,%
\begin{align*}
III_{B}(\zeta_{t}^{N})=\int_{t_{1}}^{t_{2}}N_{B}(t)\text{ } & (\phi
_{1}(X_{B}(t),t)+g(X_{B}(t),\zeta_{t}^{N})\ \phi_{1}(X_{B}(t),t)- \\
& -\mu(X_{B}(t),\zeta_{t}^{N})\phi(X_{B}(t),t))dt.
\end{align*}
This term is shown to be equivalent with%
\[
\int_{t_{1}}^{t_{2}}\frac{d}{dt}\left( N_{B}(t)\text{ }\phi(X_{B}(t),t)%
\right) -\phi(X_{B}(t),t)\sum_{i=B}^{N}\beta(X_{i}(t),\zeta_{t}^{N})\
N_{i}(t)dt. 
\]
Using the original definition for the boundary cohort dynamics, %
\eqref{N_deRoos_definition} and \eqref{pi_deRoos_definition}, we derive the
required correction term%
\begin{gather*}
\int_{t_{1}}^{t_{2}}\frac{d}{dt}\left( N_{B}(t)\text{ }\phi(x_{B}(t),t)%
\right) -\phi(X_{B}(t),t)\sum_{i=B}^{N}\beta(X_{i}(t),\zeta_{t}^{N})\
N_{i}(t)dt-III_{B}(\zeta_{t}^{N})= \\
=\int_{t_{1}}^{t_{2}}N_{B}(t)\left( \dfrac{dX_{B}(t)}{dt}-g(x_{B}(t),%
\zeta_{t}^{N})\ \right) \text{ }\phi_{1}(X_{B}(t),t)+ \\
+\mu_{1}(X_{B}(t),\zeta_{t}^{N})\pi_{B}(t)\phi(X_{B}(t),t)dt.
\end{gather*}
\end{proof}

By Lemma \ref{bounded_terms_original_def}, we see that the correction term
above is bounded by $C(t_{2}-t_{1})^{2}$. Analogous to Lemma \ref%
{Lemma_step_4}, we then have

\begin{lemma}
\label{Lemma_step_4_original_definition}With $\zeta_{t}^{N}$ defined by the
EBT method with internalizations at times $t_{i}=iT/n$ , we have that%
\[
R_{\phi}(\zeta_{t}^{N},\nu_{0},0,T)\rightarrow0, 
\]
as $N$ and $n$ tends to infinity. Here $\nu_{0}$ is the initial data at time 
$t=t_{0}=0$.
\end{lemma}

The original definition of the boundary cohorts might prove more challenging
from a numerical perspective. However, if we can determine numerically
solutions $\zeta_{t}^{N,h}$ to the equations of the EBT method such that the
center of mass, $X_{B}^{h}$, now determined by the non linear transformation %
\eqref{nonlinear_transformation} converges to its true value, $X_{B}$, as
the step length $h\searrow0$, the residual still tends to zero according to
Lemma \ref{numerical_approximation}. Hence, the numerical convergence
follows as before.

\section{Discussion}

Enhanced biological realism and predictive ability of theoretical
investigations are gaining importance as anthropogenic impacts are
fundamentally altering the native environment of many organisms.
Physiologically structured population models (PSPMs) are increasingly used
to model and analyze biological systems. As these models account for the
physiological development of individuals, they are better able to predict
system dynamics. In contrast to simple unstructured population models such
as the classical Lotka-Volterra equations, PSPMs often defy analytical
investigations due to the non-local dependencies. There is thus a mounting
need for numerical methods that can effectively uncover the underlying
dynamics. The Escalator Boxcar Train (EBT) has been specifically designed
for PSPMs and has three major advantages:\ it prevents numerical diffusion,
it is relatively easy to implement, and the underlying equations allow for a
natural biological interpretation. The method was developed more than two
decades ago and has been used to study PSPMs ever since, but the fact that
convergence has never been formally proved might well have hampered its
wider acceptance beyond the domains of theoretical biology.

In this paper we have given the first rigorous proof of convergence for the
EBT method. Our proof is given in a modern setting of measure-valued
solutions (see e.g., \cite{Gwiazda2010}). This contrasts with previous
efforts by de Roos and Metz \cite{Roos1991} that were carried out in a
classical setting and thus required additional smoothness assumptions. While
their efforts fell short of proving the full convergence of the EBT method,
the authors succeeded in showing that the method consistently approximates
the true solution, i.e., that the local approximation error as measured
through an arbitrary (but smooth) functional of the solution is bounded and
vanishes in the limit of infinitely fine discretization of the individual
state space.

There are many possible extensions of the work presented here. A
straightforward extension is to write down the corresponding proof for a
higher-dimensional state space but with a single birth state. We believe
that with more tedious calculations, one could prove the convergence also
for the case of stochastic birth state. A more challenging extension is to
consider stochasticity in individual development. On the population-level,
this roughly amounts to diffusion and it is difficult to see how the EBT
method\ should best be adapted to deal with this situation. Here, some
inspiration might come from moving-mesh discontinuous Galerkin methods
which, at least at first glance, appear to have similarities with the EBT
method. A further extension is to consider different formulations of the
boundary cohort. We initially proved convergence when the boundary cohort
differed only by the addition of a fecundity term. While this works
mathematically, it is natural to account for the fact that newborn
individuals reduces the average size of individuals in the boundary cohort.
The original formulation of the EBT method does account for this through a
different definition of the boundary cohort, and as a second step we
analyzed and proved convergence for this case. We believe that our proof can
be extended to show convergence also for other formulations of the boundary
cohort, as long as the flux of individuals is preserved. Analyzing
convergence rates for different definitions of boundary cohorts would be an
interesting extension of the work presented here. In particular, we believe
that the series expansion around the size at birth underlying the original
derivation of the boundary cohorts is not required, and that a direct
evaluation at the center of mass might lead to even faster convergence. This
could well be part of a more broadly encompassing study that explores
convergence rates under different smoothness assumptions. A final important
extension would be to consider vital rates that depend on the entire history
of the population state up to the current time, rather than merely the
current population state, as this would encompass cases with dynamic
environmental feedback variables.

Given the long tradition of partial differential equations (PDEs)\ in the
physical sciences, it is not surprising that PSPMs were initially studied
using this\ formalism. Efforts in the last decades have revealed, however,
that the PDE\ formalism is not well-suited for considering questions of
existence, uniqueness, and stability. For this reason, the cumulative
formulation of structured population models \cite{Diekmann1998,Diekmann2001}
was developed. It had the drawback, however, that a principle of linearized
stability and the Hopf bifurcation theorem proved hard to establish \cite%
{gyllenberg_mathematical_2007}. Currently, it appears that renewal equations
are well-suited for studying PSPMs \cite%
{diekmann_stability_2008,diekmann_second_2008,diekmann_daphnia_2009,gyllenberg_mathematical_2007}%
. The work presented here has been developed from the PDE setting. We
believe, however, that renewal equations are a promising framework for
developing and analyzing numerical methods for PSPMs. A\ first step would be
to recast the EBT method\ in this setting, after which the extensions
outlined above could be considered. With interest in PSPMs now mounting, a
historical opportunity exists for bridging biological theory and
computational mathematics through the development of modern numerical
methods for the 21st century.

\begin{acknowledgement}
\AA .B. and D.S. gratefully acknowledge support from the Kempe Foundations.
We thank Odo Diekmann, Mats Larson, and Hans Metz for valuable comments and
suggestions.
\end{acknowledgement}


\begin{thebibliography}{99}
\bibitem{Bogachev2007} V.~I. Bogachev. \newblock {\em Measure theory. {V}ol.
{I}, {II}}. \newblock Springer-Verlag, Berlin, 2007.

\bibitem{briggs_dynamical_1995} C.~J. Briggs, R.~M. Nisbet,
W.~W. Murdoch, T.~R. Collier, and {J.A.J.} Metz. \newblock %
Dynamical effects of {Host-Feeding} in parasitoids. \newblock {\em Journal
of Animal Ecology}, 64(3):403--416, 1995. 

\bibitem{Carrillo2012} J.~A.~Carrillo, R.~M.~Colombo, P.~Gwiazda, P., and A.~Ulikowska. \newblock Structured populations, cell growth and measure valued balance laws. {\em J. Differential  Eqautions} 252(4):3245--3277, 2012.

\bibitem{Roos1988} A.~M. de~Roos. \newblock Numerical methods for structured
population models: the escalator boxcar train. \newblock {\em Numer. Methods
Partial Differential Equations}, 4(3):173--195, 1988.

\bibitem{Roos1997} A.~M. de~Roos. \newblock A gentle introduction to models
of physiologically structured populations. \newblock In S.~Tuljapurkar and
H.~Caswell, editors, \emph{Structured-population models in marine,
terrestrial, and freshwater systems}, pages 119--204. Chapman \& Hall, New
York, 1997.

\bibitem{Roos1991} A.~M. de~Roos and J.~A.~J. Metz. \newblock Towards a
numerical analysis of the escalator boxcar train. \newblock In \emph{%
Differential equations with applications in biology, physics, and
engineering ({L}eibnitz, 1989)}, volume 133 of \emph{Lecture Notes in Pure
and Appl. Math.}, pages 91--113. Dekker, New York, 1991.

\bibitem{Diekmann2005} O.~Diekmann and Ph.~Getto. \newblock Boundedness,
global existence and continuous dependence for nonlinear dynamical systems
describing physiologically structured populations. \newblock {\em J.
Differential Equations}, 215(2):268--319, 2005.

\bibitem{diekmann_second_2008} O.~Diekmann and M.~Gyllenberg. \newblock The
second {Half-With} a quarter of a century delay. \newblock {\em Mathematical
Modelling of Natural Phenomena}, 3(7):36--48, October 2008.

\bibitem{Diekmann2001} O.~Diekmann, M.~Gyllenberg, H.~Huang, M.~Kirkilionis,
J.~A.~J. Metz, and H.~R. Thieme. \newblock On the formulation and analysis
of general deterministic structured population models. {II}. {N}onlinear
theory. \newblock {\em J. Math. Biol.}, 43(2):157--189, 2001.

\bibitem{diekmann_stability_2008} O.~Diekmann, Ph.~Getto, and M.~Gyllenberg. \newblock Stability and bifurcation analysis of volterra
functional equations in the light of suns and stars. \newblock {\em {SIAM}
Journal on Mathematical Analysis}, 39(4):1023, 2008.

\bibitem{diekmann_daphnia_2009} O.~Diekmann, M.~Gyllenberg, J.~A.~J.
Metz, S.~Nakaoka, and A.~M. Roos. \newblock Daphnia revisited: local
stability and bifurcation theory for physiologically structured population
models explained by way of an example. \newblock {\em J.
Math. Biol.}, 61(2):277--318, 2009.

\bibitem{Diekmann1998} O.~Diekmann, M.~Gyllenberg, J.~A.~J.~Metz, and
H.~R.~Thieme. \newblock On the formulation and analysis of general
deterministic structured population models. {I}. {L}inear theory. \newblock
{\em J. Math. Biol.}, 36(4):349--388, 1998.

\bibitem{goetz_using_2008} R.~Goetz, N.~ Hritonenko, A.~Xabadia,
and Y.~Yatsenko. \newblock Using the escalator boxcar train to determine
the optimal management of a {Size-Distributed} forest when carbon
sequestration is taken into account. \newblock In Ivan Lirkov, Svetozar
Margenov, and Jerzy Wasniewski, editors, \emph{{Large-Scale} Scientific
Computing}, volume 4818 of \emph{Lecture Notes in Computer Science}, 
334--341. Springer Berlin / Heidelberg, 2008.

\bibitem{Gwiazda2010} P.~Gwiazda, T.~Lorenz, and A.~Marciniak-Czochra. \newblock A nonlinear structured population model: {L}%
ipschitz continuity of measure-valued solutions with respect to model
ingredients. \newblock {\em J. Differential Equations}, 248(11):2703--2735,
2010.

\bibitem{Gwiazda2010b} P.~Gwiazda and A.~Marciniak-Czochra. \newblock Structured population equations in metric spaces. 
\newblock {\em J. Hyperbolic Differ. Equ.} 7(4):733--773, 2010.

\bibitem{gyllenberg_mathematical_2007} M.~Gyllenberg. \newblock %
Mathematical aspects of physiologically structured populations: the
contributions of j. a. j. metz. \newblock {\em J. Biological
Dynamics}, 1(1):3--44, 2007.

\bibitem{Hormander_1990_book} L. H{\"o}rmander. \newblock {\em The
analysis of linear partial differential operators. {I}}. \newblock Springer
Study Edition. Springer-Verlag, Berlin, second edition, 1990. 

\bibitem{Metz1986} J.~A.~J. Metz and O.~Diekmann. \newblock {\em Formulating
models for structured populations}, volume~68 of \emph{Lecture Notes in
Biomath.} \newblock Springer, Berlin, 1986.

\bibitem{persson_ontogenetic_1998} L.~Persson, K.~Leonardsson, A.~M.~de~Roos, M.~Gyllenberg, and B.~Christensen. \newblock %
Ontogenetic scaling of foraging rates and the dynamics of a {Size-Structured}
{Consumer-Resource} model. \newblock {\em Theor. Popul. Biol.},
54(3):270--293, 1998.

\bibitem{xabadia_optimal_2010} A.~Xabadia and R.~U. Goetz. \newblock %
The optimal selective logging regime and the faustmann formula. \newblock
{\em Journal of Forest Economics}, 16(1):63--82, 2010.
\end{thebibliography}
\end{document}